\documentclass[a4paper,12pt]{article}

\usepackage{amscd}
\usepackage{amssymb,amsfonts,amsmath,amsthm}
\usepackage{verbatim}

\usepackage{amsmath}
\usepackage{amsthm}
\usepackage{amssymb}

\usepackage{latexsym}
\usepackage{array}
\usepackage{enumerate}

\numberwithin{equation}{section}

\usepackage{amsmath,amssymb,mathrsfs,amsthm}
\usepackage[all]{xy}
\usepackage{graphicx}
\usepackage{ulem}
\usepackage{ascmac}
\usepackage{mathtools}

\newcommand{\BDC}{{\mathbf{D}}^{\mathrm{b}}}

\newcommand{\Mod}{\mathrm{Mod}}

\newcommand{\CC}{\mathbb{C}}

\newcommand{\RR}{\mathbb{R}}

\newcommand{\ZZ}{\mathbb{Z}}
\newcommand{\D}{\mathcal{D}}

\newcommand{\F}{\mathcal{F}}
\newcommand{\G}{\mathcal{G}}
\newcommand{\M}{\mathcal{M}}

\newcommand{\sho}{\mathcal{O}}

\newcommand{\PP}{{\mathbb P}}

\newcommand{\an}{{\rm an}}

\newcommand{\e}{\varepsilon}

\newcommand{\Perv}{{\rm Perv}}
\newcommand{\Sol}{{\rm Sol}}

\newcommand{\tl}[1]{\widetilde{#1}}

\newcommand{\simto}{\overset{\sim}{\longrightarrow}}

\newcommand{\op}{\mbox{\scriptsize op}}
\newcommand{\SD}{\mathcal{D}}

\newcommand{\SO}{\mathcal{O}}

\newcommand{\SM}{\mathcal{M}}
\newcommand{\SN}{\mathcal{N}}

\newcommand{\SE}{\mathcal{E}}
\newcommand{\SF}{\mathcal{F}}
\newcommand{\SG}{\mathcal{G}}

\newcommand{\Modcoh}{\mathrm{Mod}_{\mbox{\rm \scriptsize coh}}}
\newcommand{\Modhol}{\mathrm{Mod}_{\mbox{\rm \scriptsize hol}}}
\newcommand{\Modrh}{\mathrm{Mod}_{\mbox{\rm \scriptsize rh}}}

\newcommand{\BDCcoh}{{\mathbf{D}}^{\mathrm{b}}_{\mbox{\rm \scriptsize coh}}}

\newcommand{\BDChol}{{\mathbf{D}}^{\mathrm{b}}_{\mbox{\rm \scriptsize hol}}}
\newcommand{\BDCrh}{{\mathbf{D}}^{\mathrm{b}}_{\mbox{\rm \scriptsize rh}}}
\newcommand{\DD}{\mathbb{D}}

\newcommand{\Dotimes}{\overset{D}{\otimes}}

\newcommand{\Potimes}{\overset{+}{\otimes}}

\newcommand{\rhom}{{\rm R}{\mathcal{H}}om}
\newcommand{\rihom}{{\rm R}{\mathcal{I}}hom}
\newcommand{\Prihom}{{\rm R}{\mathcal{I}}hom^+}
\newcommand{\Prhom}{\rhom^+}
\newcommand{\I}{{\rm I}}
\newcommand{\che}[1]{\overset{\vee}{#1}}
\newcommand{\var}[1]{\overline{#1}}
\newcommand{\BEC}{{\mathbf{E}}^{\mathrm{b}}}
\newcommand{\Q}{\mathbf{Q}}
\newcommand{\EE}{\mathbb{E}}

\newcommand{\T}{{\rm T}}
\newcommand{\bfR}{\mathbf{R}}
\newcommand{\bfL}{\mathbf{L}}
\newcommand{\bfD}{\mathbf{D}}
\newcommand{\rmR}{{\rm R}}
\newcommand{\rmE}{{\rm E}}

\newcommand{\rmt}{{\rm t}}
\newcommand{\bfE}{\mathbf{E}}
\newcommand{\rmL}{{\rm L}}
\renewcommand{\Re}{\rm Re}

\newtheorem{theorem}{Theorem}[section]

\newtheorem{corollary}[theorem]{Corollary}
\newtheorem{lemma}[theorem]{Lemma}
\newtheorem{proposition}[theorem]{Proposition}

\theoremstyle{definition}

\theoremstyle{remark}

\title{On Some Topological Properties of 
Fourier Transforms of Regular Holonomic $\SD$-Modules 
}

\author{ Yohei ITO 
\footnote{Graduate School of Mathematical Science, The University of 
Tokyo, 3-8-1, Komaba, 
Meguro, Tokyo, 153-8914, Japan. 
E-mail: yitoh@ms.u-tokyo.ac.jp } 
and Kiyoshi TAKEUCHI 
\footnote{Institute of Mathematics, University  of 
Tsukuba, 1-1-1, Tennodai, 
Tsukuba, Ibaraki, 305-8571, Japan. 
E-mail: takemicro@nifty.com} }

\date{}

\begin{document}
\maketitle

\begin{abstract}
We study Fourier transforms of regular holonomic 
$\SD$-modules. In particular we show that 
their solution complexes are monodromic. 
An application to direct images of 
some irregular holonomic $\SD$-modules 
will be given. Moreover we give a new proof to 
the classical theorem of Brylinski and 
improve it by showing its converse. 
\end{abstract}

\section{Introduction}\label{sec:1}
First of all 
we recall Fourier transforms of algebraic $\SD$-modules. 
Let $X=\CC_z^N$ be a complex vector space 
and $Y=\CC_w^N$ its dual. 
We regard them as algebraic varieties and 
use the notations $\SD_X$ and $\SD_Y$ for 
the rings of ``algebraic'' differential operators on them. 
Denote by $\Modcoh(\SD_X)$ (resp. $\Modhol(\SD_X)$) 
the category of coherent (resp. holonomic) $\SD_X$-modules. 
Let $W_N := \CC[z, \partial_z]\simeq\Gamma(X; \SD_X)$ and 
$W^\ast_N := \CC[w, \partial_w]\simeq\Gamma(Y; \SD_Y)$ 
be the Weyl algebras over $X$ and $Y$, respectively. 
Then by the ring isomorphism 
\begin{equation*}
W_N\simto W^\ast_N\hspace{30pt}
(z_i\mapsto-\partial_{w_i},\ \partial_{z_i}\mapsto w_i) 
\end{equation*}
we can endow a left $W_N$-module $M$ with 
a structure of a left $W_N^\ast$-module.
We call it the Fourier transform of $M$ 
and denote it by $M^\wedge$. 
For a ring $R$ we denote by $\Mod_f(R)$ 
the category of finitely generated $R$-modules. 
Recall that for the affine algebraic varieties $X$ 
and $Y$ we have the equivalences of categories 
\begin{align*} 
\Modcoh(\SD_X)
&\simeq
\Mod_f(\Gamma(X; \SD_X)) = \Mod_f(W_N),\\
\Modcoh(\SD_Y)
&\simeq
\Mod_f(\Gamma(Y; \SD_Y)) = \Mod_f(W^\ast_N)
\end{align*}
(see e.g. \cite[Propositions 1.4.4 and 1.4.13]{HTT08}).
For a coherent $\SD_X$-module $\SM\in\Modcoh(\SD_X)$
we thus can define its Fourier 
transform $\SM^\wedge\in\Modcoh(\SD_Y)$. 
It follows that we obtain an equivalence of categories
\begin{equation*}
( \cdot )^\wedge : \Modhol(\SD_X)\simto \Modhol(\SD_Y)
\end{equation*}
between the categories of holonomic $\SD$-modules.
However the Fourier transform $\SM^\wedge$ of a regular 
holonomic $\SD_X$-module $\SM$ is 
not necessarily regular. For the regularity of 
$\SM^\wedge$ we need some strong condition on $\SM$. 
Recall that a constructible sheaf 
$\SF \in\BDC_{\CC-c}(\CC_{X})$ on $X= \CC^N$ 
is called monodromic if 
its cohomology sheaves are locally constant 
on each $\CC^*$-orbit in $X= \CC^N$. 
Note that this condition was introduced by 
Verdier in \cite{Ver83}. 
Then the following beautiful theorem 
is due to Brylinski \cite{Bry86}. 

\begin{theorem}[Brylinski \cite{Bry86}]\label{th-Bry}
Let $\SM$ be an algebraic regular holonomic 
$\SD$-module on $X= \CC^N$. Assume that its 
solution complex $Sol_{X}(\SM)$ is 
monodromic. Then its 
Fourier transform $\SM^\wedge$ is regular 
and $Sol_{Y}(\SM^\wedge)$ is monodromic. 
\end{theorem}

Recently in \cite{IT18} the authors studied the 
Fourier transforms of general regular 
holonomic $\SD$-modules very precisely by 
using the Riemann-Hilbert correspondence for 
irregular holonomic $\SD$-modules established by 
D'Agnolo and Kashiwara \cite{DK16} and 
the Fourier-Sato transforms 
for enhanced ind-sheaves developed by 
Kashiwara and Schapira \cite{KS16-2}. 
In this process we found a new proof 
of Theorem \ref{th-Bry} (see the proof of 
Theorem \ref{blythm}). Recall that Brylinski 
proved it by reducing the problem to 
the case $N=1$ and using some deep results 
on nearby cycle $\SD$-modules. Our new 
proof is purely geometric and relies on 
the Riemann-Hilbert correspondence of 
D'Agnolo and Kashiwara \cite{DK16}. 
See the proof of Theorem \ref{blythm} for the details. 
In our study of Fourier transforms of regular 
holonomic $\SD$-modules, we found also that for a 
regular holonomic $\SD_X$-module $\SM$ 
the enhanced solution complex 
$Sol_Y^\rmE ( \SM^\wedge )$ of its 
Fourier transform $\SM^\wedge$ satisfies 
a special condition. More precisely, 
for a $\RR_+$-conic sheaf 
$\SG \in \BDC(\CC_{Y \times \RR})$ on 
$Y \times \RR \simeq \RR^{2N+1}$ we found 
an isomorphism 
\[ Sol_Y^\rmE ( \SM^\wedge ) \simeq 
\CC_Y^\rmE \Potimes \SG, \]
where we regard $\SG$ as an ind-sheaf on 
the bordered space $Y \times \RR_\infty$ 
by the natural embedding 
\[\xymatrix@C=30pt@M=10pt{\BDC( 
\CC_{Y \times \RR})\ar@{^{(}->}[r] & 
\BDC(\I\CC_{Y \times \RR_{\infty}}).}\]
See Corollary \ref{cor1}. 
From this we obtain the following result. 

\begin{theorem}\label{newthm2}
Let $\SM$ be an algebraic regular holonomic 
$\SD$-module on $X= \CC^N$. 
Then $Sol_Y(\M^\wedge)$ is monodromic.
\end{theorem}

It seems that this result is already implicit 
in the main theorem of Daia \cite{Dai00}. 
Indeed, for regular holonomic $\M\in\Modrh(\D_X)$ it implies that 
$Sol_Y(\M^\wedge)$ is $\RR_+$-conic. Note that 
the recent result \cite[Lemma 6.1.3]{DHMS17} of 
D'Agnolo, Hien, Morando and Sabbah 
implies also the same property of $Sol_Y(\M^\wedge)$ 
(see also \cite[Lemma 1.5.2]{DHMS17}). 
For a general theory of conic ind-sheaves see \cite{Pre11}.
The monodromicity of $Sol_Y(\M^\wedge)$ in Theorem \ref{newthm2} 
follows from its $\CC$-constructibility and 
the $\RR_+$-conicness (see Lemma \ref{lem1}). 
In this paper we prove Theorem \ref{newthm2} by using 
the theory of enhanced ind-sheaves and 
our results in \cite{IT18}. 
In this way, we can also improve Brylinski's 
Theorem \ref{th-Bry} as follows. 

\begin{corollary}
Let $\SM$ be an algebraic regular holonomic 
$\SD$-module on $X= \CC^N$. 
Then $\M^\wedge$ is regular if and only if $\M$ is monodromic.
\end{corollary}

Namely we prove the converse of Brylinski's 
theorem. Moreover, as a simple application of 
Theorem \ref{newthm2} we obtain the following 
result which may be of independent interest. 

\begin{theorem}\label{ntheorem-3}
Let $\rho : X=\CC^N\twoheadrightarrow Z=\CC^{N-1}$
be a surjective linear map of codimension one and 
$\SM$ an algebraic regular holonomic 
$\SD$-module on $X= \CC^N$. 
For the dual $L\simeq\CC^{N-1}$ of $Z$ let 
$\iota : L\xhookrightarrow{\ \ \ } Y=\CC^N$ 
be the injective linear map induced by $\rho$. 
Then for any point $a\in Y\setminus\iota(L)$ the direct image 
$\bfD\rho_\ast(\M\Dotimes\sho_Xe^{- \langle z, a\rangle})
\in\BDChol(\D_Z)$ is concentrated in degree $0$.
\end{theorem}

Recently many mathematicians studied 
direct images of irregular holonomic 
$\SD$-modules and obtained precise 
results. For example, see Heizinger \cite{Hei15}, 
Hien-Roucairol \cite{HR08} and 
Roucairol \cite{Rou06}, \cite{Rou07} etc. 
Note also that in the case $N=1$ 
Fourier transforms of general 
holonomic $\SD$-modules were precisely 
studied by many authors such as 
Bloch-Esnault \cite{BE04}, 
D'Agnolo-Kashiwara \cite{DK17}, 
Mochizuki \cite{Mochi10, Mochi18} and 
Sabbah \cite{Sab08} etc.

\section{Preliminary Notions and Results}\label{uni-sec:2}
In this section, we briefly recall some basic notions 
and results which will be used in this paper. 
We assume here that the reader is familiar with the 
theory of sheaves and functors in the framework of 
derived categories. For them we follow the terminologies 
in \cite{KS90} etc. For a topological space $M$ 
denote by $\BDC(\CC_M)$ the derived category 
consisting of bounded 
complexes of sheaves of $\CC$-vector spaces on it. 
The following lemma will be used in 
the proofs of Theorems \ref{blythm} and \ref{thm2} . 

\begin{lemma}\label{lem1}
Assume that a $\CC$-constructible sheaf 
$\G\in\BDC_{\CC-c}(\CC_{\CC^N})$
on $\CC^N$ is $\RR_+$-conic.
Then it is monodromic.
\end{lemma}

\begin{proof}
By restrictions we may assume that $N=1$. 
By the $\CC$-constructibility of $\G$ there exists a finite subset
$\{P_1, P_2, \ldots, P_k\}\subset\CC$ of $\CC\simeq\RR^2$ 
such that $(H^j\G)|_ {\CC\setminus\{P_1, P_2, \ldots, P_k\}}$ 
is a local system for any $j\in\ZZ$. 
For $1\leq i\leq k$ such that $P_i\neq0$ 
let $\ell_i=\RR_+P_i\simeq\RR_+$ be the real half line in 
$\CC\simeq\RR^2$ passing through the point $P_i$. 
Then by our assumption $(H^j\G)|_{\ell_i}$ is a 
constant sheaf for any $j\in\ZZ$. 
This implies that for the function $h_i : \CC\to\CC, h_i(x)=x-P_i$ such that
$h_i^{-1}(0)=\{P_i\}\subset\CC$ we have $\phi_{h_i}(\G)\simeq0$,
where 
\[ \phi_{h_i} : 
\BDC_{\CC-c}(\CC_\CC)\to  \BDC_{\CC-c}(\CC_{h_i^{-1}(0)})
\]
is Deligne's vanishing cycle functor. 
From now we shall use an argument in Sabbah 
\cite[\S 8]{Sab06}. Let 
\[ {}^p \phi_{h_i} = \phi_{h_i} [-1] : 
\BDC_{\CC-c}(\CC_\CC)\to  \BDC_{\CC-c}(\CC_{h_i^{-1}(0)})
\]
be the perverse (or shifted) vanishing cycle functor. 
Recall that it preserves the perversity. 
For $j\in\ZZ$ let ${}^pH^j(\G)\in\Perv(\CC)$ be the 
$j$-th perverse cohomology sheaf of $\G$.
Then ${}^p \phi_{h_i}({}^pH^j(\G))$ is 
concentrated in degree $0$ 
for any $j\in\ZZ$. 
Hence there exists an isomorphism 
$H^j({}^p \phi_{h_i}(\G)) \simeq H^0({}^p \phi_{h_i}({}^pH^j(\G)))$
for any $j\in\ZZ$. 
We thus obtain ${}^p \phi_{h_i}({}^pH^j(\G))\simeq 0$ for any 
$1\leq i\leq k$ and $j\in\ZZ$. 
This shows that the perverse sheaves 
${}^pH^j(\G)$'s are smooth on $\CC^\ast$ i.e. 
$H^l ({}^pH^j(\G))|_{\CC^\ast}$ 
is a local system on $\CC^\ast$ for any $j, l \in \ZZ$.
Then the assertion immediately follows.
\end{proof}

\subsection{Ind-sheaves}\label{sec:3}
We recall some basic notions 
and results on ind-sheaves. References are made to 
Kashiwara-Schapira \cite{KS01} and \cite{KS06}. 
Let $M$ be a good topological space (which is locally compact, 
Hausdorff, countable at infinity and has finite soft dimension). 
We denote by $\Mod(\CC_M)$ the abelian category of sheaves 
of $\CC$-vector spaces on it and by $\I\CC_M$ 
that of ind-sheaves. Then there exists a 
natural exact embedding $\iota_M : \Mod(\CC_M)\to\I\CC_M$ 
of categories. We sometimes omit it.
It has an exact left adjoint $\alpha_M$, 
that has in turn an exact fully faithful
left adjoint functor $\beta_M$: 
 \[\xymatrix@C=60pt{\Mod(\CC_M)  \ar@<1.0ex>[r]^-{\iota_{M}} 
 \ar@<-1.0ex>[r]_- {\beta_{M}} & \I\CC_M 
\ar@<0.0ex>[l]|-{\alpha_{M}}}.\]

The category $\I\CC_M$ does not have enough injectives. 
Nevertheless, we can construct the derived category $\BDC(\I\CC_M)$ 
for ind-sheaves and the Grothendieck six operations among them. 
We denote by $\otimes$ and $\rihom$ the operations 
of tensor products and internal homs respectively. 
If $f : M\to N$ be a continuous map, we denote 
by $f^{-1}, \rmR f_\ast, f^!$ and $\rmR f_{!!}$ the 
operations of inverse images,
direct images, proper inverse images and proper direct images 
respectively. 
We set also $\rhom := \alpha_M\circ\rihom$. 
Note that $(f^{-1}, \rmR f_\ast)$ and 
$(\rmR f_{!!}, f^!)$ are pairs of adjoint functors.

\subsection{Ind-sheaves on Bordered Spaces}\label{sec:4} 
For the results in this subsection, we refer to 
D'Agnolo-Kashiwara \cite{DK16}. 
A bordered space is a pair $M_{\infty} = (M, \che{M})$ of
a good topological space $\che{M}$ and an open subset $M\subset\che{M}$.
A morphism $f : (M, \che{M})\to (N, \che{N})$ of bordered spaces
is a continuous map $f : M\to N$ such that the first projection
$\che{M}\times\che{N}\to\che{M}$ is proper on
the closure $\var{\Gamma}_f$ of the graph $\Gamma_f$ of $f$ 
in $\che{M}\times\che{N}$. 
The category of good topological spaces embeds into that
of bordered spaces by the identification $M = (M, M)$. 
We define the triangulated category of ind-sheaves on 
$M_{\infty} = (M, \che{M})$ by 
\[\BDC(\I\CC_{M_\infty}) := 
\BDC(\I\CC_{\che{M}})/\BDC(\I\CC_{\che{M}\backslash M}).\]
Let 
\[\mathbf{q} : \BDC(\I\CC_{\che{M}})\to\BDC(\I\CC_{M_\infty})\]
be the quotient functor. 
For a morphism $f : M_\infty\to N_\infty$ 
of bordered spaces, 
we have the Grothendieck's operations 
\begin{align*} 
\otimes &: \BDC(\I\CC_{M_\infty})\times
\BDC(\I\CC_{M_\infty})\to\BDC(\I\CC_{M_\infty}), \\
\rihom &: \BDC(\I\CC_{M_\infty})^{\op}\times
\BDC(\I\CC_{M_\infty})\to\BDC(\I\CC_{M_\infty}), \\
\rmR f_\ast &: \BDC(\I\CC_{M_\infty})\to\BDC(\I\CC_{N_\infty}),\\
f^{-1} &: \BDC(\I\CC_{N_\infty})\to\BDC(\I\CC_{M_\infty}),\\
\rmR f_{!!} &: \BDC(\I\CC_{M_\infty})\to\BDC(\I\CC_{N_\infty}),\\
f^! &: \BDC(\I\CC_{N_\infty})\to\BDC(\I\CC_{M_\infty})
\end{align*}
(see \cite[Definitions 3.3.1 and 3.3.4]{DK16}).
Moreover, there exists a natural embedding 
\[\xymatrix@C=30pt@M=10pt{\BDC(\CC_M)\ar@{^{(}->}[r] & 
\BDC(\I\CC_{M_\infty}).}\]

\subsection{Enhanced Sheaves}\label{sec:5}
For the results in this subsection, see 
Tamarkin \cite{Tama08}, 
Kashiwara-Schapira \cite{KS16-2} and 
D'Agnolo-Kashiwara \cite{DK17}. 
Let $M$ be a good topological space. 
We consider the maps 
\[M\times\RR^2\xrightarrow{p_1, p_2, \mu}M
\times\RR\overset{\pi}{\longrightarrow}M\]
where $p_1, p_2$ are the first and the second projections 
and we set $\pi (x,t):=x$ and 
$\mu(x, t_1, t_2) := (x, t_1+t_2)$. 
Then the convolution functors for 
sheaves on $M \times \RR$ are defined by
\begin{align*}
F_1\Potimes F_2 &:= \rmR \mu_!(p_1^{-1}F_1\otimes p_2^{-1}F_2),\\
\Prhom(F_1, F_2) &:= \rmR p_{1\ast}\rhom(p_2^{-1}F_1, \mu^!F_2).
\end{align*}
We define the triangulated category of enhanced sheaves on $M$ by 
\[\BEC(\CC_M) := \BDC(\CC_{M\times\RR})/\pi^{-1}\BDC(\CC_M). \] 
Let 
\[\Q : \BDC(\CC_{M \times \RR} )  \to\BEC(\CC_M)\]
be the quotient functor. The convolution functors 
are defined also for enhanced sheaves. We denote them 
by the same symbols $\Potimes$, $\Prhom$. 
For a continuous map $f : M \to N $, we 
can define naturally the operations 
$\bfE f^{-1}$, $\bfE f_\ast$, $\bfE f^!$, $\bfE f_{!}$ 
for enhanced sheaves. 
We have also a 
natural embedding $\e : \BDC( \CC_M) \to \BEC( \CC_M)$ defined by 
\[ \e(F) := \Q(\CC_{\{t\geq0\}}\otimes\pi^{-1}F). \] 

For a continuous function $\varphi : U\to \RR$ 
defined on an open subset $U \subset M$ of $M$  we define 
the exponential enhanced sheaf by 
\[ {\rm E}_{U|M}^\varphi := 
\Q(\CC_{\{t+\varphi\geq0\}} ), \]
where $\{t+\varphi\geq0\}$ stands for 
$\{(x, t)\in M\times{\RR}\ |\ x\in U, t+\varphi(x)\geq0\}$.

\subsection{Enhanced Ind-sheaves}\label{sec:6}
We recall some basic notions 
and results on enhanced ind-sheaves. References are made to 
D'Agnolo-Kashiwara \cite{DK16} and 
Kashiwara-Schapira \cite{KS16}. 
Let $M$ be a good topological space.
Set $\RR_\infty := (\RR, \var{\RR})$ for 
$\var{\RR} := \RR\sqcup\{-\infty, +\infty\}$,
and let $t\in\RR$ be the affine coordinate. 
Then we define the triangulated category 
of enhanced ind-sheaves on $M$ by 
\[\BEC(\I\CC_M) := \BDC(\I\CC_{M 
\times\RR_\infty})/\pi^{-1}\BDC(\I\CC_M), \]
where $\pi : M\times\RR_\infty\to M$ is a morphism
of bordered spaces induced by the first projection $M\times\RR\to M$.
The quotient functor
\[\Q : \BDC(\I\CC_{M\times\RR_\infty})\to\BEC(\I\CC_M)\]
has fully faithful left and right adjoints 
$\bfL^\rmE,\bfR^\rmE$ defined by 
\[\bfL^\rmE(\Q K) := (\CC_{\{t\geq0\}}\oplus\CC_{\{t\leq 0\}})
\Potimes K ,\hspace{20pt} \bfR^\rmE(\Q K) :
=\Prihom(\CC_{\{t\geq0\}}\oplus\CC_{\{t\leq 0\}}, K), \]
where $\{t\geq0\}$ stands for 
$\{(x, t)\in M\times\var{\RR}\ |\ t\in\RR, t\geq0\}$ 
and $\{t\leq0\}$ is defined similarly.

We consider the maps 
\[M\times\RR_\infty^2\xrightarrow{p_1, p_2, \mu}M
\times\RR_\infty\]
where $p_1$ and $p_2$ are morphisms
of bordered spaces induced by the projections.
And $\mu$ is a morphism of bordered spaces induced by the map 
$M\times\RR^2\ni(x, t_1, t_2)\mapsto(x, t_1+t_2)\in M\times\RR$.
Then the convolution functors for 
ind-sheaves on $M \times \RR_\infty$ are defined by
\begin{align*}
F_1\Potimes F_2 &:= \rmR\mu_{!!}(p_1^{-1}F_1\otimes p_2^{-1}F_2),\\
\Prihom(F_1, F_2) &:= \rmR p_{1\ast}\rihom(p_2^{-1}F_1, \mu^!F_2).
\end{align*}
The convolution functors 
are defined also for enhanced ind-sheaves. We denote them 
by the same symbols $\Potimes$, $\Prihom$. 
For a continuous map $f : M \to N $, we 
can define also the operations 
$\bfE f^{-1}$, $\bfE f_\ast$, $\bfE f^!$, $\bfE f_{!!}$ 
for enhanced ind-sheaves. For example, 
by the natural morphism $\tl{f}: M \times \RR_{\infty} \to 
N \times \RR_{\infty}$ of bordered spaces associated to 
$f$ we set $\bfE f_\ast ( \Q K)= \Q\big(\rmR \tl{f}_{\ast}(K)\big)$. 
The other operations are defined similarly. 
We thus obtain the six operations $\Potimes$, $\Prihom$,
$\bfE f^{-1}$, $\bfE f_\ast$, $\bfE f^!$, $\bfE f_{!!}$ 
for enhanced ind-sheaves. 
Set $\CC_M^\rmE := \Q 
\Bigl(``\underset{a\to +\infty}{\varinjlim}"\ \CC_{\{t\geq a\}}
\Bigr)\in\BEC(\I\CC_M)$. 
Then we have 
natural embeddings $\e, e : \BDC(\I\CC_M) \to 
\BEC(\I\CC_M)$ defined by 
\begin{align*}
\e(F) & := \Q(\CC_{\{t\geq0\}}\otimes\pi^{-1}F) \\ 
e(F) &:=  \CC_M^\rmE\otimes\pi^{-1}F
\simeq \CC_M^\rmE\Potimes\e(F). 
\end{align*}

For a continuous function $\varphi : U\to \RR$ 
defined on an open subset $U \subset M$ of $M$  we define 
the exponential enhanced ind-sheaf by 
\[\EE_{U|M}^\varphi := 
\CC_M^\rmE\Potimes {\rm E}_{U|M}^\varphi
=
\CC_M^\rmE\Potimes
\Q\CC_{\{t+\varphi\geq0\}}
, \]
where $\{t+\varphi\geq0\}$ stands for 
$\{(x, t)\in M\times\var{\RR}\ |\ t\in\RR, x\in U, 
t+\varphi(x)\geq0\}$.

\subsection{$\SD$-Modules}\label{sec:7}
In this subsection we recall some basic notions 
and results on $\SD$-modules. 
References are made to 
\cite{HTT08}, \cite[\S 7]{KS01}, 
\cite[\S 8, 9]{DK16} and 
\cite[\S 3, 4, 7]{KS16}.  
For a complex manifold $X$ we denote by 
$d_X$ its complex dimension. 
Denote by $\SO_X$ and $\SD_X$ 
the sheaves of holomorphic functions 
and holomorphic differential operators on $X$ respectively. 
Let $\BDC(\SD_X)$ be the bounded derived category 
of left $\SD_X$-modules. 
Moreover we denote by $\BDCcoh(\SD_X)$,
$\BDChol(\SD_X)$ and $\BDCrh(\SD_X)$ the full 
triangulated subcategories
of $\BDC(\SD_X)$ consisting of objects with coherent,
holonomic and regular holonomic cohomologies, 
respectively.
For a morphism $f : X\to Y$ of complex manifolds, 
denote by $\Dotimes, \rhom_{\SD_X}, \bfD f_\ast, \bfD f^\ast$, 
$\DD_X : \BDCcoh(\SD_X)^{\op} 
\simto \BDCcoh(\SD_X)$  
the standard operations for $\SD$-modules. 
The classical solution functor is defined by  
\begin{align*}
Sol_X &: \BDCcoh (\SD_X)^{\op}\to\BDC(\CC_X),
\hspace{30pt}\SM \longmapsto \rhom_{\SD_X}(\SM, \SO_X).
\end{align*}
One defines the ind-sheaf $\SO_X^\rmt$ of tempered 
holomorphic functions 
as the Dolbeault complex with coefficients in the ind-sheaf 
of tempered distributions. 
More precisely, denoting by $X^c$ the 
complex conjugate manifold 
to $X$ and by $X_{\RR}$ the underlying real analytic manifold of $X$,
we set
\[\SO_X^\rmt := \rihom_{\SD_{X^c}}(
\SO_{X^c}, \mathcal{D}b_{X_\RR}^\rmt), \]
where $\mathcal{D}b_{X_\RR}^\rmt$ is the ind-sheaf of 
tempered distributions on $X_\RR$
(for the definition see \cite[Definition 7.2.5]{KS01}). 
Then the tempered solution functor is defined by 
\begin{align*}
Sol_X^\rmt &: \BDCcoh (\SD_X)^{\op}\to\BDC(\I\CC_X), 
\hspace{40pt}
\SM \longmapsto \rihom_{\SD_X}(\SM, \SO_X^\rmt). 
\end{align*}
Note that we have isomorphisms
\begin{align*}
Sol_X(\SM) &\simeq \alpha_XSol_X^\rmt(\SM).
\end{align*}

Let $i : X\times\RR_\infty\to X\times\PP$ be
the natural morphism of bordered spaces and
$\tau\in\CC\subset\PP$ the affine coordinate 
such that $\tau|_\RR$ is that of $\RR$. 
We then define an object $\SO_X^\rmE\in\BEC(\I\SD_X)$ by 
\begin{align*}
\SO_X^\rmE &:= \rihom_{\SD_{X^c}}(
\SO_{X^c}, \mathcal{D}b_{X_\RR}^{\T}) 
\simeq i^!\rihom_{\SD_\PP}(\SE_{\CC|\PP}^{\tau}, 
\SO_{X\times\PP}^\rmt)[2], 
\end{align*}
where $\mathcal{D}b_{X_\RR}^{\T}$ 
stand for the enhanced ind-sheaf 
of tempered distributions on $X_\RR$ 
(for the definition see \cite[Definition 8.1.1]{DK16}). 
We call $\SO_X^\rmE$ the enhanced ind-sheaf 
of tempered holomorphic functions. 
Note that there exists an isomorphism
\[i_0^!\bfR^\rmE\SO_X^\rmE\simeq\SO_X^\rmt, \]
where $i_0 : X\to X\times\RR_\infty$ is the 
inclusion map of bordered spaces induced by $x\mapsto (x, 0)$. 
The enhanced solution functor is defined by 
\begin{align*}
Sol_X^\rmE &: \BDCcoh (\SD_X)^{\op}\to\BEC(\I\CC_X), 
\hspace{40pt} 
\SM \longmapsto \rihom_{\SD_X}(\SM, \SO_X^\rmE). 
\end{align*}
Then for $\SM\in\BDCcoh(\SD_X)$ 
we have an isomorphism
\[Sol^{\rmt}_X(\M)\simeq
i_0^!{{\bfR}^\rmE}Sol_X^{\rmE}(\M).\]
Finally, we recall the following theorem of \cite{DK16}. 
\begin{theorem}[{\cite[Theorem 9.5.3 (Irregular 
Riemann-Hilbert Correspondence)]{DK16}}]\label{cor-4}
The enhanced solution functor induces a 
fully faithful one
\[ Sol_X^\rmE : \BDC_{\rm hol} (\SD_X)^{\rm op}\to\BEC(\I\CC_X). \]
\end{theorem}

\section{Fourier Transforms of Regular Holonomic 
$\SD$-modules}\label{sec:9}
In this section we inherit the situation 
and the notations in 
Section \ref{sec:1}. Let
\[X\overset{p}{\longleftarrow}X\times 
Y\overset{q}{\longrightarrow}Y\]
be the projections. 
Then by Katz-Laumon \cite{KL85}, for an 
algebraic holonomic $\SD_X$-module 
$\SM\in\Modhol(\SD_X)$ we have an isomorphism
\[\SM^\wedge\simeq
\bfD q_\ast(\bfD p^\ast\SM\Dotimes \SO_{X\times Y}
e^{- \langle z, w \rangle }),\]
where $\bfD p^\ast, \bfD q_\ast, \Dotimes$ are
the operations for algebraic $\SD$-modules
and $\SO_{X\times Y}
e^{- \langle z, w \rangle }$ is the integral connection
of rank one on $X\times Y$ associated to the canonical
paring $\langle, \rangle : X\times Y\to\CC$.
In particular the right hand side is concentrated in degree zero. 
Let $\var{X}\simeq\PP^N$ (resp. $\var{Y}\simeq\PP^N$)
be the projective compactification of $X$ (resp. $Y$). 
By the inclusion map $i_X : 
X=\CC^N\xhookrightarrow{\ \ \ }\var{X}=\PP^N$
we extend a holonomic $\SD_X$-module $\SM\in\Modhol(\SD_X)$
on $X$ to the one $\tl{\SM} := i_{X\ast}\SM\simeq\bfD i_{X\ast}\SM$ 
on $\var{X}$.
Denote by $\var{X}^{\an}$ the underlying complex manifold
of $\var{X}$ and define the analytification
$\tl{\SM}^{{\ }\an}\in\Modhol(\SD_{\var{X}^{\an}})$
of $\tl{\SM}$ by
$\tl{\SM}^{{\ }\an} := \SO_{\var{X}^{\an}}
\otimes_{\SO_{\var{X}}}\tl{\SM}$.
Then we set 
\[Sol_{\var X}^\rmE(\tl\SM) := 
Sol_{\var{X}^{\an}}^\rmE(\tl{\SM}^{{\ }\an})
\in\BEC(\I\CC_{\var{X}^{\an}}).\]
Similarly for the Fourier transform $\SM^\wedge\in\Modhol(\SD_Y)$, 
by the inclusion map $i_Y : 
Y=\CC^N\xhookrightarrow{\ \ \ }\var{Y}=\PP^N$ 
we define $Sol_{\var Y}^\rmE(\tl{\SM^\wedge}) 
 \in\BEC(\I\CC_{\var{Y}^{\an}})$. 
Let
\[\var{X}^{\an}\overset{\var{p}}{\longleftarrow}
\var{X}^{\an}\times\var{Y}^{\an}\overset{\var{q}}{\longrightarrow}
\var{Y}^{\an}\]
be the projections. 
Then the following theorem is essentially due to
Kashiwara-Schapira \cite{KS16-2} and D'Agnolo-Kashiwara \cite{DK17}. 
For $F\in\BEC(\I\CC_{\var{X}^{\an}})$ we set
\[{}^\rmL F := \bfE\var{q}_{*}(\bfE\var{p}^{-1}F
\Potimes\EE_{X\times Y|\var{X}\times\var{Y}}^{
-{\Re} \langle  z, w \rangle }[N])
\in\BEC(\I\CC_{\var{Y}^{\an}} ) \]
(here we denote $X^{\an}\times Y^{\an}$ etc.
by $X\times Y$ etc. for short) and call 
it the Fourier-Sato (Fourier-Laplace) transform of $F$. 

\begin{theorem}\label{thm-1}
For $\SM\in\Mod_{\rm hol}(\SD_X)$ there exists an isomorphism
\[Sol_{\var Y}^\rmE(\tl{\SM^\wedge})
\simeq{}^\rmL Sol_{\var X}^\rmE(\tl{\SM}).\]
\end{theorem}

From now on, we focus our attention on Fourier transforms 
of regular holonomic $\SD_X$-modules.
For such a $\SD_X$-module $\SM$, 
by \cite[Theorem 7.1.1]{HTT08} we have 
an isomorphism $Sol_{\var X} (\tl{\SM}) 
\simeq i_{X!}Sol_X(\SM)$, where the right hand side 
$i_{X!} Sol_X(\SM)\in D^b(\CC_{\var{X}^{\an}})$ 
is the extension by zero of the classical solution complex 
of $\SM$ to $\var{X}^{\an}$. Moreover 
by \cite[Proposition 9.1.3 and Corollary 9.4.9]{DK16} 
there exists an isomorphism 
\[Sol_{\var X}^\rmE(\tl{\SM})
\simeq
\CC_{\var{X}^{\an}}^\rmE\Potimes\e(i_{X!}Sol_X(\SM)).\]
For an enhanced sheaf $F\in\BEC(\CC_{\var{X}^{\an}})$
on $\var{X}^{\an}$ we define its Fourier-Sato 
(Fourier-Laplace) transform
${}^\rmL F\in\BEC(\CC_{\var{Y}^{\an}})$ by
\[{}^\rmL F := \bfE\var{q}_{*}(\bfE\var{p}^{-1}F\Potimes
\rmE_{X\times Y|\var{X}\times\var{Y}}^{
-{\Re} \langle z, w \rangle}[N])
\in\BEC(\CC_{\var{Y}^{\an}}).\]
Since we have
\[{}^\rmL(\CC^\rmE_{\var{X}^{\an}}\Potimes( \cdot ))
\simeq
\CC^\rmE_{\var{Y}^{\an}}\Potimes{}^\rmL( \cdot ), \]
for the calculation of $Sol_{\var Y}^\rmE(\tl{\SM^\wedge})$ 
it suffices to calculate the Fourier-Sato transform of 
the enhanced sheaf $\e(i_{X!}Sol_X(\SM))
\in\BEC( \CC_{\var{X}^{\an}})$
on $\var{X}^{\an}$. 
The following theorem is due to Brylinski \cite{Bry86}. 
Here we give a new geometric proof to it. 

\begin{theorem}\label{blythm}
Let $\SM$ be an algebraic regular holonomic 
$\SD$-module on $X= \CC^N$. 
Assume that $Sol_X(\SM)$ is monodromic. Then 
$\SM^\wedge$ is also a regular holonomic $\SD_Y$-module 
and $Sol_Y(\SM^\wedge)$ is monodromic. 
\end{theorem}

\begin{proof}
By the above argument we have isomorphisms 
\begin{align*}
&Sol_{\var Y}^\rmE(\tl{\SM^\wedge})
\simeq 
{}^\rmL Sol_{\var X}^\rmE(\tl{\SM}) 
\\ 
&\simeq
{}^\rmL\Big(
\CC_{\var X^\an}^\rmE\Potimes\e\big( i_{X!} Sol_X(\SM)\big)\Big)\\
&\simeq
\CC_{\var Y^\an}^\rmE\Potimes{}^\rmL\Big(\e\big( i_{X!} 
Sol_X(\SM)\big)\Big)\\
&\simeq
\CC_{\var Y^\an}^\rmE\Potimes\e\big( i_{Y!}Sol_X(\SM)^\wedge \big), 
\end{align*}
where $( \cdot )^\wedge$ stands for 
the Fourier-Sato transform for 
$\RR_+$-conic sheaves (see \cite{KS90}) and 
in the last isomorphism we applied 
\cite[Theorem 5.7]{KS16-2} to the $\RR_+$-conic sheaf 
$Sol_X(\SM)$. 
Note that $Sol_X(\M)^\wedge$ is not only $\RR_+$-conic
but also $\CC$-constructible by 
\cite[Proposition 10.3.18]{KS90}. 
Hence it is monodromic by Lemma \ref{lem1}. 
Moreover by applying the functor 
$i_0^!{\rmR^\rmE} ( \cdot )$ to the isomorphism 
$Sol_{\var Y}^\rmE(\tl{\SM^\wedge}) \simeq 
\CC_{\var Y^\an}^\rmE\Potimes\e\big( i_{Y!}Sol_X(\SM)^\wedge \big)$ 
we obtain an isomorphism 
\[ Sol_{\var Y} (\tl{\SM^\wedge}) \simeq 
i_{Y!}Sol_X(\SM)^\wedge. \]
This implies that $i_{Y!}Sol_X(\SM)^\wedge$ is an 
(algebraic) constructible sheaf on the 
algebraic variety $\var Y$. By 
\cite[Corollary 7.2.4]{HTT08} we can take 
a regular holonomic $\SD$-module 
$\SN\in \Mod_{\rm rh}(\SD_{\var Y})$ on $\var Y$ 
such that $Sol_{\var Y}(\SN)\simeq 
i_{Y!} Sol_X(\SM)^\wedge$. 
Then we have isomorphisms 
\begin{align*}
&Sol_{\var Y}^\rmE(\tl{\SM^\wedge})\simeq
\CC_{\var Y^\an}^\rmE\Potimes\e\big( i_{Y!} Sol_X(\SM)^\wedge\big)\\
&\simeq
\CC_{\var Y^\an}^\rmE\Potimes\e\big(Sol_{\var Y}(\SN)\big)\\
&\simeq
Sol_{\var Y}^\rmE( \SN ). 
\end{align*}
By Theorem \ref{cor-4} 
we thus obtain an isomorphism 
\[ ( \tl{\SM^\wedge} )^{\an} 
\simeq \SN^{\an} \quad \in 
\Mod_{\rm rh}(\SD_{{\var Y}^{\an}}) \]
of analytic $\SD$-modules on ${\var Y}^{\an}$. 
Then the assertion follows from Lemma \ref{leman} 
of Brylinski \cite[Th\'eor\`eme 7.1]{Bry86} below. 
\end{proof}

\begin{lemma}[Brylinski {\cite[Th\'eor\`eme 7.1]{Bry86}}]\label{leman} 
Let $Z$ be a smooth projective variety. Then 
the analytification functor 
\[ ( \cdot )^{\an} : \BDC_{\rm rh}(\SD_{Z}) 
\rightarrow \BDC_{\rm rh}(\SD_{Z^{\an}})
\]
is an equivalence of categories. 
\end{lemma}

\begin{proof}
This result is due to Brylinski \cite[Th\'eor\`eme 7.1]{Bry86}. 
We shall give a new proof to it. 
Let $\BDC_{\CC-c}(\CC_{Z})$ (resp. 
$\BDC_{\CC-c}(\CC_{Z^{\an}}) $) be the 
derived category of $\CC$-constructible 
sheaves on the algebraic variety $Z$ 
(resp. the complex manifold $Z^{\an}$). 
Then we have a commutative diagram of functors 
\begin{equation}
\begin{CD}
\BDC_{\rm rh}(\SD_{Z})  @>{\sim}>> 
\BDC_{\CC-c}(\CC_{Z}) 
\\
@V{( \cdot )^{\an}}VV   @VVV
\\
\BDC_{\rm rh}(\SD_{Z^{\an}})  @>{\sim}>>  
\BDC_{\CC-c}(\CC_{Z^{\an}}), 
\end{CD}
\end{equation}
where the horizontal arrows are 
the Riemann-Hilbert correspondences of 
algebraic and analytic $\SD$-modules 
respectively (see e.g. 
\cite[Theorem 7.2.2]{HTT08}). By Chow's 
theorem the right vertical arrow 
\[ \BDC_{\CC-c}(\CC_{Z}) \rightarrow 
\BDC_{\CC-c}(\CC_{Z^{\an}}) \]
is also an equivalence of categories. 
Then the assertion immediately follows. 
\end{proof}

For $s\in\RR_+$ let 
\[m_s : Y=\CC^N \simto Y=\CC^N, \quad w\mapsto sw\]
be the multiplication by $s$.
We shall use also the morphism 
$\ell_s : Y\times\RR_\infty\to Y\times\RR_\infty$ 
on the bordered space $Y\times\RR_\infty$ 
induced by the diagonal action 
\[\ell_s : Y\times \RR \simto Y\times \RR, 
\quad (w, t)\mapsto (sw, st).\]
Let $f : X\times Y\times\RR\to X$, 
$g : X\times Y\times \RR \to Y\times\RR$
be the projections. Then the following lemma 
was obtained in (the proof) of 
Ito-Takeuchi \cite[Theorem 4.4]{IT18}. 

\begin{lemma}\label{lem3}
Let $\F\in\BDC(\CC_X)$. 
Then we have an isomorphism
\[{}^{\rmL}(\varepsilon ( i_{X!} \F ))\simeq
\Q \big( \tl{i_{Y}}_!{\rmR}g_!(\CC_{\{t-{\Re}\langle z, w\rangle\geq0\}}
\otimes f^{-1}\F)[N] \big) \]
of enhanced sheaves.
\end{lemma}

For $\F\in\BDC(\CC_X)$ let us set 
\[ L( \F ) = {\rmR}g_!(\CC_{\{t-{\Re}\langle z, w\rangle\geq0\}}
\otimes f^{-1}\F)[N] \in \BDC(\CC_{Y \times \RR}). 
\]

\begin{lemma}\label{prop1}
Let $\F\in\BDC(\CC_X)$.
Then for any $s\in\RR_+$
we have an isomorphism
\[\ell_s^{-1}\big( L( \F )
\big) \simeq  L( \F )\]
in $\BDC(\CC_{Y\times\RR})$. 
In other words, $L( \F )$ is a 
$\RR_+$-conic sheaf on 
$Y \times \RR \simeq \RR^{2N+1}$. 
\end{lemma}

\begin{proof}
Consider the Cartesian diagram
\[\xymatrix@M=7pt@C=70pt{
X\times Y\times \RR \ar@{->}[r]^-{{\rm id}_X\times \ell_s}
\ar@{->}[d]_-{g} & 
X\times Y\times\RR \ar@{->}[d]^{g}\\
Y\times\RR \ar@{->}[r]_-{\ell_s} & 
Y\times\RR .
}\]
Then we have isomorphisms
\begin{align*}
\ell_s^{-1}\big( L( \F ) \big)
&\simeq
\ell_s^{-1}{\rmR}g_!(\CC_{\{t-{\Re}\langle z, 
w\rangle\geq0\}}\otimes f^{-1}\F)[N]\\
&\simeq
{\rmR}g_!({\rm id}_X\times \ell_s)^{-1}
(\CC_{\{t-{\Re}\langle z, w\rangle\geq0\}}
\otimes f^{-1}\F)[N]\\
&\simeq
{\rmR}g_!(\CC_{\{t-{\Re}\langle z, w\rangle\geq0\}}
\otimes f^{-1}\F)[N]\\
&\simeq
L( \F )
\end{align*}
where in the third isomorphism we used
\[st-{\Re}\langle z, sw\rangle\geq0
\quad \Longleftrightarrow \quad 
t-{\Re}\langle z, w\rangle\geq0\]
and $f \circ({\rm id}_X\times \ell_s)=f$.
\end{proof}

From now we shall consider the special case where 
$\F=Sol_X(\M)\in\BDC(\CC_X)$ for $\M\in\Modrh(\D_X)$.
\begin{corollary}\label{cor1}
Let $\M\in\Modrh(\D_X)$ be an algebraic regular holonomic $\D_X$-module on $X$.
Then there exists a $\RR_+$-conic sheaf 
$\SG \in \BDC(\CC_{Y \times \RR})$ on 
$Y \times \RR \simeq \RR^{2N+1}$ such that
\[ Sol_Y^\rmE ( \SM^\wedge ) \simeq 
\CC_Y^\rmE \Potimes \Q(\SG). \]
\end{corollary}
\begin{proof}
By Lemma \ref{lem3} we have isomorphisms
\begin{align*}
Sol_{\var Y}^{\rmE}(\tl{\M^\wedge})
&\simeq
\CC_{\var Y}^{\rmE}\Potimes{}^{\rmL}\varepsilon ( i_{X!} Sol_X(\M))\\
&\simeq
\CC_{\var Y}^{\rmE}\Potimes \Q(\tl{i_{Y}}_!{L}(Sol_X(\M))).
\end{align*}
Then by the restriction to $Y\subset \var Y$ and 
Lemma \ref{prop1} we obtain the assertion.
\end{proof}
\begin{proposition}\label{prop3}
Let $\M\in\Modrh(\D_X)$ be an algebraic regular holonomic $\D_X$-module on $X$.
Then for any $s\in\RR_+$ we have an isomorphism
\[\ell_s^{-1}{\rmL}^{\rmE}Sol_{Y}^{\rmE}({\M^\wedge})\simeq
{\rmL}^{\rmE}Sol_{Y}^{\rmE}({\M^\wedge}).\]
\end{proposition}
\begin{proof}
There exist isomorphisms
\begin{align*}
{\rmL}^{\rmE}Sol_{\var Y}^{\rmE}(\tl{\M^\wedge})
&\simeq
{\rmL}^{\rmE}\big(\CC_{\var Y}^{\rmE}\Potimes \Q(\tl{i_{Y}}_!{L}(Sol_X(\M)))\big)\\
&\simeq
\CC_{\{t\gg0\}}\Potimes \tl{i_{Y}}_!{L}(Sol_X(\M)).
\end{align*}
We extend $\ell_s$ to $\var Y\times\RR$ naturally and denote it 
by the same symbol.
Then by Lemma \ref{prop1} for $s\in\RR_+$ we have isomorphisms
\begin{align*}
\ell_s^{-1}{\rmL}^{\rmE}Sol_{\var Y}^{\rmE}(\tl{\M^\wedge})
&\simeq
\ell_s^{-1}(\CC_{\{t\gg0\}}\Potimes \tl{i_{Y}}_!{L}(Sol_X(\M)))\\
&\simeq
\CC_{\{t\gg0\}}\Potimes\ell_s^{-1}( \tl{i_{Y}}_!{L}(Sol_X(\M)) )\\
&\simeq
\CC_{\{t\gg0\}}\Potimes \tl{i_{Y}}_!{L}(Sol_X(\M))\\
&\simeq
{\rmL}^{\rmE}Sol_{\var Y}^{\rmE}(\tl{\M^\wedge}).
\end{align*}
We obtain the assertion by the restriction to $Y\subset \var Y$.
\end{proof}

\begin{proposition}\label{prop2}
Let $\SM$ be an algebraic regular holonomic 
$\SD$-module on $X= \CC^N$. 
Then for any $s\in\RR_+$ we have an isomorphism
\[m_s^{-1}Sol_Y(\M^\wedge)\simeq Sol_Y(\M^\wedge).\]
\end{proposition}

\begin{proof}
By (the proof) of 
Ito-Takeuchi \cite[Lemma 3.13]{IT18} 
there exist isomorphisms
\begin{align*}
Sol_Y(\M^\wedge)&\simeq
\alpha_Yi_0^!{\rmR}^{\rmE}(Sol_{Y}^{\rmE}({\M^\wedge}))\\
&\simeq
\alpha_Y{\rmR}\pi_{\ast}\rihom(\CC_{\{t\geq0\}}
\oplus\CC_{\{t\leq0\}},
{\rmL}^{\rmE}Sol_{Y}^{\rmE}({\M^\wedge})).
\end{align*}
Consider the commutative diagram
\[\xymatrix@M=7pt@C=70pt{
Y\times \RR_\infty\ar@{->}[r]^-{\pi}
\ar@{->}[d]_-{\ell_s} & 
Y\ar@{->}[d]^{m_s}\\
Y\times\RR_\infty\ar@{->}[r]_-{\pi} & 
Y.}\]
It is easy to see that it is Cartesian.
Then by Proposition \ref{prop3} we have isomorphisms
\begin{align*}
m_s^{-1}Sol_Y(\M^\wedge)&\simeq
m_s^{-1}\alpha_Y{\rmR}\pi_{\ast}
\rihom(\CC_{\{t\geq0\}}\oplus\CC_{\{t\leq0\}},
{\rmL}^{\rmE}Sol_{ Y}^{\rmE}({\M^\wedge}))\\
&\simeq
\alpha_Ym_s^!{\rmR}\pi_{\ast}
\rihom(\CC_{\{t\geq0\}}\oplus\CC_{\{t\leq0\}},
{\rmL}^{\rmE}Sol_{ Y}^{\rmE}({\M^\wedge}))\\
&\simeq
\alpha_Y{\rmR}\pi_{\ast}\ell_s^!
\rihom(\CC_{\{t\geq0\}}\oplus\CC_{\{t\leq0\}},
{\rmL}^{\rmE}Sol_{ Y}^{\rmE}({\M^\wedge}))\\
&\simeq
\alpha_Y{\rmR}\pi_{\ast}
\rihom(\ell_s^{-1}(\CC_{\{t\geq0\}}\oplus\CC_{\{t\leq0\}}),
\ell_s^!{\rmL}^{\rmE}Sol_{ Y}^{\rmE}({\M^\wedge}))\\
&\simeq
\alpha_Y{\rmR}\pi_{\ast}
\rihom(\CC_{\{t\geq0\}}\oplus\CC_{\{t\leq0\}},
{\rmL}^{\rmE}Sol_{ Y}^{\rmE}({\M^\wedge}))\\
&\simeq
Sol_Y(\M^\wedge), 
\end{align*}
where in the fifth isomorphism we used 
$\ell_s^! \simeq \ell_s^{-1}$ 
(see \cite[Corollary 3.3.11]{DK16}). 
\end{proof}

\begin{theorem}\label{thm2}
Let $\SM$ be an algebraic regular holonomic 
$\SD$-module on $X= \CC^N$. 
Then $Sol_Y(\M^\wedge)$ is monodromic.
\end{theorem}

\begin{proof}
Since the Fourier transform $\M^\wedge$ of 
$\M$ is also holonomic, 
$\Sol_Y(\M^\wedge)$ is $\CC$-constructible.
Moreover it is $\RR_+$-conic by Proposition \ref{prop2}.
Then the assertion follows from Lemma \ref{lem1}. 
\end{proof}

By this theorem we can improve Brylinski's 
Theorem \ref{blythm} as follows.
\begin{corollary}
Let $\SM$ be an algebraic regular holonomic 
$\SD$-module on $X= \CC^N$. 
Then $\M^\wedge$ is regular if and only if $\M$ is monodromic.
\end{corollary}

\begin{proof}
By Theorem \ref{blythm} the Fourier transform 
$\M^\wedge$ is regular if $\M$ is monodromic. 
It suffices to show the converse. 
Assume that $\M^\wedge$ is regular. Let 
\begin{equation*}
( \cdot )^{\vee} : \Modhol(\SD_Y)\simto \Modhol(\SD_X)
\end{equation*}
be the inverse Fourier transform. 
Then by Theorems \ref{blythm} and \ref{thm2} the 
original regular holonomic $\D_X$-module 
$\M\simeq (\M^\wedge)^{\vee}$
is monodromic. 
\end{proof}

\section{An Application to Direct Images of 
$\SD$-Modules}\label{sec:10}

In this section, we apply our results to 
direct images of some 
irregular holonomic $\SD$-modules. 
We inherit the situation and the notations in 
Section \ref{sec:1}. 
For a point $a\in Y=\CC^N$ let $\tau_a : 
Y\simto Y, w\mapsto w+a$
be the translation by it.

\begin{lemma}\label{lem2}
For $\M\in\Modcoh(\D_X)$ and $a\in Y=\CC^N$
we have an isomorphism
\[\bfD\tau_a^\ast(\M^\wedge)
\simeq
(\M\Dotimes\sho_Xe^{-\langle z, a\rangle})^\wedge.\]
\end{lemma}

\begin{proof}
By Katz-Laumon \cite{KL85}
there exist isomorphisms 
\begin{align*}
(\M\Dotimes\sho_Xe^{-\langle z, a\rangle})^\wedge
&\simeq
\bfD q_{\ast}\big(\bfD p^\ast
(\M\Dotimes \sho_Xe^{-\langle z, a\rangle})
\Dotimes\sho_{X\times Y}e^{-\langle z, w\rangle}\big)\\
&\simeq
\bfD q_{\ast}\big(\bfD p^\ast \M
\Dotimes\sho_{X \times Y} e^{-\langle z, w+a\rangle}\big)\\
&\simeq
\bfD\tau_a^\ast(\M^\wedge).
\end{align*}
\end{proof}

\begin{theorem}\label{thm3}
Let $\rho : X=\CC^N\twoheadrightarrow Z=\CC^n$
be a surjective linear map and 
$\SM$ an algebraic regular holonomic 
$\SD$-module on $X= \CC^N$. 
For the dual $L\simeq\CC^n$ of $Z$ let 
$\iota : L\xhookrightarrow{\ \ \ } Y=\CC^N$ 
be the injective linear map induced by $\rho$.
Assume that for a point $a\in Y\setminus\iota(L)$ 
the affine linear subspace
$K=\tau_a(\iota(L))\subset Y=\CC^N$ is non-characteristic for the
Fourier transform $\M^\wedge\in\Modhol(\D_Y)$ of $\M$.
Then the direct image 
$\bfD\rho_\ast(\M\Dotimes\sho_Xe^{- \langle z, a\rangle})
\in\BDChol(\D_Z)$ is concentrated in degree $0$.
\end{theorem}

\begin{proof}
Let $i_L = \iota : L\xhookrightarrow{\ \ \ }Y=\CC^N$
and $i_K : K\xhookrightarrow{\ \ \ }Y=\CC^N$
be the inclusion maps.
Then via the identification $L\simeq K$ induced  
by the translation $\tau_a$ we have isomorphisms
 \begin{align*}
 \bfD i_K^\ast(\M^\wedge)
 &\simeq
 \bfD i_L^\ast\bfD\tau_a^{\ast}(\M^\wedge)\\
 &\simeq
 \bfD i_L^\ast(\M\Dotimes\sho_X e^{-\langle z, a\rangle})^\wedge\\
 &\simeq
 \big(\bfD\rho_\ast(\M\Dotimes\sho_Xe^{-\langle z, a\rangle})\big)^\wedge,
 \end{align*}
 where in the second (resp. third) isomorphism we used
 Lemma \ref{lem2} (resp. \cite[Proposition 3.2.6]{HTT08}).
 By our assumption, the left hand side 
 $\bfD i_K^\ast(\M^\wedge)\in\BDChol(\D_K)$ 
 is concentrated in degree $0$.
 Then the assertion follows from the fact that 
 the Fourier transform is an exact functor.
\end{proof}

\begin{corollary}
In the situation of Theorem \ref{thm3}
assume also that $n=N-1$ i.e. the surjective linear map 
$\rho :  X=\CC^N\twoheadrightarrow Z=\CC^n$
is of codimension one.
Then for any point $a\in Y\setminus\iota(L)$
the direct image 
$\bfD\rho_\ast(\M\Dotimes \sho_X e^{-\langle z, a\rangle})
\in\BDChol(\D_Z)$
is concentrated in degree $0$. 
\end{corollary}

\begin{proof}
By Theorem \ref{thm2}
the Fourier transform $\M^\wedge$ of $\M$ is monodromic.
Since the affine linear subspace $K=\tau_a(\iota(L))\subset Y=\CC^N$
does not contain the origin $0\in Y=\CC^N$,
this implies that $K$ is non-characteristic for $\M^\wedge$.
Then the assertion follows  from Theorem \ref{thm3}.
\end{proof}

\end{document}